\documentclass[11pt, reqno]{article}

\usepackage{amsfonts}
\usepackage{amsmath}
\usepackage{amsthm}
\usepackage{amssymb}
\usepackage{lastpage}
\usepackage{fancyhdr}
\usepackage{fancybox}
\usepackage{mathrsfs}

\def\b{\mathbb }
\def\phi{\varphi }
\def\epsilon{\varepsilon}

\swapnumbers

\theoremstyle{plain}
\newtheorem{theorem}{Theorem}[section]
\newtheorem{corollary}[theorem]{Corollary}
\newtheorem{lemma}[theorem]{Lemma}
\newtheorem{proposition}[theorem]{Proposition}
\theoremstyle{definition}
\newtheorem{definition}[theorem]{Definition}
\newtheorem{remark}[theorem]{Remark}

\numberwithin{equation}{section}

\numberwithin{equation}{section}

\begin{document}

\title{Central Limit Theorems for Radial Random Walks on $p\times q$ Matrices
  for $p\to\infty$}
\author{Michael Voit\\
Fachbereich Mathematik, Universit\"at Dortmund\\
          Vogelpothsweg 87\\
          D-44221 Dortmund, Germany\\
e-mail: michael.voit@math.uni-dortmund.de}

\maketitle

\abstract
{Let $\nu\in M^1([0,\infty[)$ be a fixed  probability measure.
 For each dimension $p\in\b N$, let $(X_n^p)_{n\ge1}$ be 
i.i.d.~$\b R^p$-valued radial random variables with
 radial distribution $\nu$. We derive two central limit theorems for
$ \|X_1^p+\ldots+X_n^p\|_2$ for $n,p\to\infty$ with normal limits. The first CLT for $n>>p$ follows  from  
known
estimates of convergence in the CLT on $\b R^p$,
 while the second CLT for $n<<p$ will be a consequence of asymptotic properties of Bessel convolutions.

Both limit theorems are considered also for $U(p)$-invariant random walks on the space of $p\times q$
 matrices instead of   $\b R^p$ for $p\to\infty$ and fixed dimension $q$.}


Classification: 60F05, 60B10, 60B12, 33C70, 43A62

\section{Two central limit theorems}

This paper has its origin in the  following problem:
Let $\nu\in M^1([0,\infty[)$ be a fixed  probability measure. Then for each 
dimension $p\in\b N$ there is a unique radial probability measure $\nu_p\in
M^1(\b R^p)$ with $\nu$ as its radial part, i.e., $\nu$ is the image of $\nu_p$
under the norm mapping
$\phi_p(x):=\|x\|_2$. For each $p\in\b N$ consider i.i.d.~$\b R^p$-valued random
variables $X_k^p$, $k\in\b N$, with law $\nu_p$ as well as the associated radial
random walks 
$$\bigl(S_n^p:= \sum_{k=1}^n X_k^p)_{n\ge0}$$
on $\b R^p$. The  aim is to find limit
theorems for the $ [0,\infty[$-valued random variables $ \|S_n^{p}\|_2$ for
$n,p\to\infty$.
In \cite{V1} and \cite{RV} we proved that for all sequences $p_n\to\infty$,
$$ \| S_n^{p_n}\|_2^2/n\to  \sigma^2:=\sigma^2(\nu) :=\int_0^\infty x^2 \> d\nu(x)$$
under the condition $\sigma^2<\infty$. Moreover, in \cite{RV} an associated
strong law and a large deviation principle were derived under
the condition that $p_n$ grows fast enough. 
In this paper we present two associated central limit theorems (CLTs) under 
disjoint growth conditions for $p_n$.

The first CLT holds for $p_n<<n$ and is an obvious consequence of Berry-Esseen
estimates on $\b R^p$ with explicit constants depending on the dimensions $p$,
 which are due to Bentkus and G\"otze \cite{B}, \cite{BG2} (for a survey about
this topic we also recommend \cite{BGPR}):

\begin{theorem}\label{CLT1}
Assume that $\nu\in M^1([0,\infty[)$ with $\nu\ne\delta_0$ admits a finite third moment $m_3(\nu):=\int_{0}^\infty x^3 \>
d\nu(x)<\infty$, and that $\lim_{n\to\infty} n/p_n^3 =\infty $. Then
$$\frac{\sqrt p_n}{n\sigma^2\sqrt 2}(\|S_n^{p_n}\|_2^2 -n\sigma^2)$$
tends in distribution for $n\to\infty$ to the standard normal distribution 
$N(0,1)$. 
\end{theorem}

\begin{proof}
The radial measure $\nu_p$ on $\b R^p$ has a covariance matrix $\Sigma^2$ which
is invariant under all conjugations w.r.t~orthogonal transformations. Therefore,
 $\Sigma^2=c_pI_p$ with the identity $I_p$ and some constant $c_p$. As
$\sigma^2=E(\|X_1^p\|_2^2)=pc_p$, we actually have $\Sigma^2= (\sigma^2/p)I_p$. 
Theorem 2 of Bentkus \cite{B}  implies after normalization that the
distribution function $F_{n,p}$ of $\frac{p}{n\sigma^2}\|S_n^p\|_2^2$ and the
distribution function $F_p$ of the $\chi_p^2$-distribution with $p$ degrees of
freedom satisfy
\begin{equation}\label{distributionfunction}
\|F_{n,p} -F_p\|_\infty\le C\cdot\frac{p^{3/2}}{\sqrt n}
\end{equation}
for $n,p\in\b N$ with a universal $C=C(\nu)$.
Therefore, for $p=p_n$ as in the theorem, we have uniform convergence of 
 distribution functions.
Moreover, the classical CLT shows that for $\chi^2_p$-distributed random
variables $X_p$ (with $E(X_p)=p$ and $Var(X_p)=2p$), the random variables $\frac{X_p-p}{\sqrt{2p}}$ tend to the
standard normal distribution $N(0,1)$ for $p\to\infty$. A combination of both
results readily implies the theorem. 
\end{proof}

\begin{remark}\label{BG-equation-remark}
The main result of \cite{BG2} suggests that for sufficiently large dimensions $p$
and $\nu\in M^1([0,\infty[)$ with finite fourth moment $m_4(\nu):=\int_{0}^\infty x^4 \>
d\nu(x)<\infty$, 
\begin{equation}\label{BG-equation}
\|F_{n,p} -F_p\|_\infty\le C\cdot\frac{p^{2}}{ n} \quad\quad (n,p\in\b N)
\end{equation}
holds (the dependence of the constants  is not clearly noted in  \cite{BG2}
and  difficult to verify). If (\ref{BG-equation}) is true, then Theorem \ref{CLT1} holds under the weaker condition
$\lim_{n\to\infty} n/p_n^2 =\infty $.
We also remark that the results of \cite{BG1} indicate that the method of the proof
of Theorem \ref{CLT1} above cannot go much beyond  this condition.
\end{remark}

In this paper, we  derive  the
following complementary CLT for $p_n>>n$:

\begin{theorem}\label{CLT2}
Assume that $\nu\in M^1([0,\infty[)$ admits a finite fourth moment $m_4(\nu):=\int_{0}^\infty x^4 \>
d\nu(x)<\infty$, and that $\lim_{n\to\infty} n^2/p_n\to 0$. Then
$$\frac{\|S_n^{p_n}\|_2^2 -n\sigma^2}{\sqrt n}$$
tends in distribution for $n\to\infty$ to the normal distribution 
$N(0,m_4(\nu)-\sigma^4)$ on $\b R$.
\end{theorem}

\begin{remark}
\begin{enumerate}\itemsep=-1pt
\item[\rm{(1)}] Assume  that $\nu\in M^1([0,\infty[)$ admits a finite fourth moment. A simple calculation then 
yields the moments up to order 4 where in particular
\begin{equation}\label{kappa1}
E\bigl(\bigl(\|S_n^{p_n}\|_2^2 -n\sigma^2\bigr)^2\bigr) = n(m_4(\nu)-\sigma^4) +2\frac{n(n-1)}{p_n}\sigma^4.
\end{equation}
Theses moments up to order 4  lead to the {\bf conjecture} that  the assertion of Theorem
  \ref{CLT1} holds precisely for  $n/p_n\to \infty$, and the assertion of Theorem  \ref{CLT2}
precisely for $n/p_n\to 0$. 
In fact, this was recently proved for measures having all moments by the moment convergence method
by Grundmann \cite{G}. He also obtains results for the case $p_n=nc$.
\item[\rm{(2)}] 
A comparison of  Theorems \ref{CLT1} and \ref{CLT2} has the following possible implication
to statistics: Assume that $\nu\in M^1([0,\infty[)$ is known and that the random variable
$\|S_n^{p_n}\|_2$
can be observed with a known time parameter $n$, but an unknown dimension $p$
which has to be estimated. 
Then $p$ can be recovered in a reasonable way for $n>>p^3$ while this is
not the case for $n<<\sqrt p$.
\end{enumerate}\end{remark}

In this paper we shall also derive  two  generalizations of the preceding CLTs:

\bigskip 
The first extension concerns a matrix-valued version:
For fixed dimensions  $p,q\in \b N$ let 
$M_{p,q} = M_{p,q}(\b F)$  be the space of $p\times q$-matrices over 
 $\b F = \b R, \b C$ or the quaternions $\b H$ with real dimension $d=1,2$ or
 $4$ respectively.
 This is a Euclidean vector space of real dimension $dpq$ with scalar product
$\langle x,y\rangle = \mathfrak R {tr}(x^*y)$ where $x^* := \overline x^t$, 
$\mathfrak R t := \frac{1}{2}(t+ \overline t)$ is the real part of $t\in \b F$,  and
${tr}$  is the trace in $M_{q}:= M_{q,q}.$
A  measure  on $M_{p,q}$ is called radial if it is 
invariant under the action of the unitary group $U_p= U_p(\b F)$ 
by left multiplication,
$U_p\times M_{p,q} \to M_{p,q}$,  $(u,x) \mapsto ux$.
This action is  orthogonal w.r.t.~the scalar product above, and, by uniqueness
of the polar decomposition, two matrices $x,y\in M_{p,q}$
belong to  the same $U_p$-orbit if and only if $x^*x = y^*y$. 
Thus the  space $M_{p,q}^{U_p}$ of $U_p$-orbits in $M_{p,q}$ is naturally parameterized by the 
cone $\Pi_q = \Pi_{q}(\b F)$ of positive semidefinite $q\times q$-matrices over
$\b F$. We
 identify $M_{p,q}^{U_p}$ with  $\Pi_q$ via $U_qx\simeq (x^*x)^{1/2}$, i.e., the
 canonical projection $M_{p,q}\to M_{p,q}^{U_p}$ will be realized as the mapping
$$ \phi_p: M_{p,q} \to \Pi_{q},\quad x\mapsto (x^*x)^{1/2}.$$ 
The square root is used here in order to ensure for  $q=1$ and  $\b F= \b R$
that the
 setting above with
 $\Pi_1=[0,\infty[$ and $\phi_p(x)=\|x\|$ appears.
By  taking images of measures,   $\phi_p$ induces a Banach space
isomorphism between the space $M_b^{U_q}(M_{p,q})$  of all bounded radial 
Borel measures on $M_{p,q}$ and  the space $M_b(\Pi_q)$ of bounded Borel measures on the cone $\Pi_q$.  In particular, for
each  $\nu\in M^1(\Pi_q)$ there is a unique radial probability
measure $\nu_p\in M^1(M_{p,q})$ with $\phi_p(\nu_p)=\nu$. 

As in the case $q=1$, we now consider for each $p\in\b N$  i.i.d.~$M_{p,q}$-valued random
variables $X_k^p$, $k\in\b N$, with law $\nu_p$ and the associated radial
random walks 
$\bigl(S_n^p:= \sum_{k=1}^n X_k^p)_{n\ge0}$.

\begin{definition}
We  say that $\nu\in M^1(\Pi_q)$ admits a $k$-th 
moment ($k\in\b N$) if 
$$m_k(\nu):=\int_{\Pi_q}\|s\|^k \> d\nu(s)<\infty$$ where  $\|s\|=
(tr{s^2})^{1/2}$ is the Hilbert-Schmidt norm. 
 If the second moment exists, the second moment of $\nu$ is defined as the matrix-valued integral
$$\sigma^2:=\sigma^2(\nu):=\int_{\Pi_q} s^2 \> d\nu(s)\in \Pi_q.$$
\end{definition}

With these notions,  the following generalizations of 
 Theorems \ref{CLT1} and \ref{CLT2} hold:

\begin{theorem}\label{CLT3}
Assume that  $m_4(\nu)<\infty$. Moreover, let $\lim_{n\to\infty} n/p_n^4 =\infty $. Then
$$\frac{\sqrt{p_n}}{n}(\phi_{p_n}(S_n^{p_n})^2-n\sigma^2)$$
tends in distribution  to some normal distribution $N(0,T^2)$ on the
vector space $H_q$ of hermitian $q\times q$-matrices over $\b F$ (with a
 covariance matrix $T^2=T^2(\sigma^2)$ described in the proof below for $\b F=\b R$). 
\end{theorem}

\begin{proof}
We regard $M_{p,q}= M_{p,q}(\b F)$ as $\b F^p\otimes \b F^q$.  
The radial measure $\nu_p$ on $M_{p,q} $ has a covariance matrix $\Sigma^2_p$ which
is invariant under all conjugations w.r.t.~$U_p$, i.e., we have
 $\Sigma^2_p=I_p\otimes T_p$ for some $T_p\in \Pi_q$. As
$\sigma^2=E((X_k^p)^*X_k^p)=pT_p$, we have $\Sigma^2_p= \frac{1}{p}\cdot I_p\otimes\sigma^2 $. 

Moreover, Theorem 1 of Bentkus \cite{B}  implies after normalization that there is an universal constant $C>0$ with
$$\bigl| P(\frac{\sqrt p}{\sqrt n}S_n^p\in K)- N(0,\Sigma^2_1)(K)|\le C \cdot\frac{p^{2}}{\sqrt n}$$
for all convex sets $K\subset M_{p,q}$ and all $n,p$.
Therefore,  for $p=p_n$ as in the theorem, 
 $$\bigl|P(\frac{\sqrt{p_n}}{\sqrt n}S_n^{p_n}\in K)- N(0,\Sigma^2_1)(K)\bigr|\to0  \quad\quad{\rm for}\quad
n\to\infty$$
 uniformly in all convex sets $K\subset M_{p,q}$. 
Using the projections $ \phi_p: M_{p,q} \to \Pi_{q}$, we obtain 
\begin{equation}\label{limit-esti} |P(\frac{p_n}{n}\cdot\phi_{p_n}(S_n^{p_n})^2\in L)- W_{p_n}(L)|\to0
 \quad\quad{\rm for}\quad n\to\infty \end{equation}
 uniformly in all convex sets $L\subset \Pi_q$
where the measures $W_{p_n}:=\phi_{p_n}^2(N(0,\Sigma^2_1))$ are certain Wishart distributions on $\Pi_q$ with $p_n$ degrees of freedom.
By definition, the  $W_{p_n}$ appear as the distribution of a $p_n$-fold sum of  iid $ W_{1}$-distributed  random variables on 
the vector space $H_q$ with expectation $\sigma^2$ and some covariance matrix $T^2=T^2(\sigma^2)$ described below.
A combination of Eq.~(\ref{limit-esti}) and the classical CLT on $H_q$ then readily implies the theorem. 

We finally compute $T^2$  for  $\b F=\b R$.
Let $X=(X_1,\ldots,X_q)$ be a $\b R^q$-valued, standard normal distributed random variable,
 and $\sigma\in\Pi_q$ the positive semidefinite root of $\sigma^2$. Then $Y:=X\sigma$ is $\b R^q$-valued with distribution $N(0,\sigma^2)$, and
the $H_q$-valued random variable $Y^*Y=X^*\sigma^2X$ has the distribution $ W_{1}$. 
We notice that $E(X_i^4)=3$, $E(X_i^2X_j^2)=1$ for $i\ne j$, and that $E(X_iX_jX_kX_l)=0$ whenever 
at least one index appears only once in $\{i,j,k,l\}$. This implies 
\begin{align}
&(T^2)_{(i,j),(k,l)} = Kov(Y_iY_j, Y_kY_l)= E(Y_iY_jY_kY_l)-E(Y_iY_j)E(Y_kY_l)
\notag\\ 
&= \sum_{a,b,c,d=1}^q \sigma_{a,i} \sigma_{b,j} \sigma_{c,k} \sigma_{d,l}E(X_aX_bX_cX_d) 
 -\Bigl(\sum_{a=1}^q\sigma_{a,i} \sigma_{a,j}\Bigr)\Bigl(\sum_{a=1}^q\sigma_{a,k} \sigma_{a,l}\Bigr)
\notag\\ 
&= \sum_{a,b=1}^q \bigl(\sigma_{a,i}\sigma_{a,k}\sigma_{b,j}\sigma_{b,l} + \sigma_{a,i}\sigma_{a,l}\sigma_{b,j}\sigma_{b,k}\bigr)
\notag\\ 
&= (\sigma^2)_{i,k}(\sigma^2)_{j,l} +(\sigma^2)_{i,l}(\sigma^2)_{j,k}.
\end{align}
The computation for $\b F=\b C,\b H$ is similar.
\end{proof}

Notice that the proof  is analog to that of  Theorem \ref{CLT1}. The 
 slightly stronger
condition is required  as here certain convex sets in
$M_{p,q}$ instead of balls are used, where only weaker
convergence results  form \cite{B} are available.
As for $q=1$, we expect that  Theorem  \ref{CLT3} remains true under
 slightly weaker conditions than $ n/p_n^4 \to\infty $.

\begin{theorem}\label{CLT4}
Assume that $m_4(\nu)<\infty$ and 
$\lim_{n\to\infty} n^2/p_n =0 $. Then
$$\frac{1}{\sqrt n}(\phi_{p_n}(S_n^{p_n})^2 -n\sigma^2)$$
tends in distribution to the normal distribution $N(0,\Sigma^2)$ on $H_q$ where 
$\Sigma^2$ is the covariance matrix of $\phi_{p_n}(X_1^{p_n})$ (which is
independent of $p_n$).
\end{theorem}

Notice that for $q=1$, Theorem  \ref{CLT4} completely agrees with \ref{CLT2}.
Theorem  \ref{CLT4} will appear in Section 3 below as a special case of the even more general 
CLT \ref{theorem-mu-ge-n}.

\bigskip
We next turn to this generalization:
Consider again the Banach space isomorphism $\phi_p:M_b^{U_q}(M_{p,q})\to M_b(\Pi_q).$ The usual group convolution on $M_{p,q}$ induces 
a Banach-$*$-algebra-structure on $M_b(\Pi_q)$ such that this  becomes a
probability-preserving Banach-$*$-algebra isomorphism. 
 The space $\Pi_q$ together with this new convolution becomes a so-called commutative orbit hypergroup; see \cite{J},\cite{BH},
and \cite{R}. Moreover, for  $p\ge 2q$, Eq.~(3.5) and Corollary 3.2 of
\cite{R} show that  the convolution  of  point measures on
$\Pi_q$ induced from $M_{p,q}$ is given by
\begin{equation}\label{def-convo-groupcase}
(\delta_r *_\mu \delta_s)(f) := 
\frac{1}{\kappa_{\mu}}\int_{D_q} f\bigl(\sqrt{r^2 + s^2 + svr + rv^*\!s}\,\bigr)\,
\Delta(I-vv^*)^{\mu-\rho}\, dv
\end{equation}
with $\mu:=pd/2$,
$\rho:=d\bigl(q-\frac{1}{2}\bigr) +1$,
$$ D_{q} := \{ v\in M_{q}: v^*v < I\}$$
(where $ v^*v <I$ means that $I- v^*v$ is  positive definite),
and with the normalization constant
\begin{equation}\label{kappa}
 \kappa_\mu := \int_{D_q} \Delta(I-v^*v)^{\mu-\rho} dv. 
\end{equation}
The convolution on $M_b(\Pi_q)$ is just given by bilinear, weakly
continuous extension.

It was observed in  \cite{R} that Eq.~(\ref{def-convo-groupcase})  defines a
commutative hypergroup $(\Pi_q, *_\mu)$ for all indices $\mu\in\b R$ with
$\mu>\rho-1$, where $0\in \Pi_q$ is the  identity  and the involution is   the identity mapping. These hypergroups are closely related 
with a product formula for Bessel functions $J_\mu$  on  the cone $\Pi_q$ and 
are therefore called Bessel hypergroups. For details  we
refer to \cite{FK}, \cite{H}, and in particular \cite{R}.
For general indices $\mu$, these Bessel hypergroups  $(\Pi_q, *_\mu)$ have no
group interpretation as in the cases $\mu=pd/2$ with integral $p$, but nevertheless the notion of random walks on these
hypergroups is still meaningful. For $q=1$, such
structures and associated random walks were investigated by Kingman \cite{K}
and many others; see \cite{BH}.

\begin{definition}\label{random-walk}
 Fix $\mu>\rho-1$ and a probability measure $\nu\in M^1(\Pi_q)$.
A Bessel random walk $(S_n^\mu)_{n\ge0}$ on $\Pi_q$ of index $\mu$ and with law $\nu$
 is a time-homogeneous Markov chain 
on $\Pi_q$   with  $S_0^\mu=0$ and   transition probability
$$P(S_{n+1}^\mu\in A| S_n^\mu=x) \> =\> (\delta_x*_\mu \nu)(A)$$
for $x\in\Pi_q$ and Borel sets $A\subset \Pi_q$. 
\end{definition}

This notion has its origin in the following well-known fact
for  the orbit cases $\mu=pd/2$, $p\in\b N$: If for a given $\nu\in
M^1(\pi_q)$ we consider the associated radial random walk $(S_n^p)_{n\ge0}$ on
$M_{p,q}$ as  above, then $(\phi_p(S_n^p))_{n\ge0}$ is a random walk  on $\Pi_q$ of index $\mu=pd/2$
with law $\nu$.

We shall derive Theorem \ref{CLT4} in Section 3 in this more general
setting for $\mu\in\b R$, $\mu\ge 2q$, as the proof is  precisely  the same as in the
group case. The proof will rely on facts on these Bessel convolutions which
we recapitulate in the next section 

We finally mention that it seems reasonable that at  least for  $q=1$, Theorem \ref{CLT1} may be also generalized
to Bessel random walks with arbitrary indices $\mu\in\b R$ with $\mu\to\infty$. A possible approach might work via explicit 
Berry-Esseen-type estimates for  Hankel transforms similar as in \cite{PV} with a careful investigation of the dependence of constants there
on the dimension parameter.

\section{Bessel convolutions on matrix cones}

In this section we collect some known facts mainly from  \cite{R} and \cite{V2}.

Let $\b F$ be one of the real division algebras $\b R, \b C$ 
or  $\b H$ with real dimension $d=1,2$ or $4$ respectively.
Denote the usual conjugation in $\b F$ by $t\mapsto \overline t$,
the real part  of $t\in \b F$ by $\mathfrak R t = \frac{1}{2}(t + \overline
t)$, and by $|t| = (t\overline t)^{1/2}$ its norm.

For $p,q\in \b N$ we denote  by $M_{p,q}$ 
the vector space of all  $p\times q$-matrices
 over $\b F$ and put $M_q:= M_q(\b F):=   M_{q,q}(\b F)$ for abbreviation. Let further
$H_q = \{x\in M_{q}: x= x^*\}$
the space of Hermitian $q\times q$-matrices. All these spaces are real Euclidean
 vector spaces with scalar product 
 $\langle x, y\rangle := \mathfrak R {tr}(x^*y)$  and  norm $\|x\|=\langle x, x\rangle^{1/2}$. Here 
$x^* := \overline x^t$ and ${tr}$ denotes the trace.
Let further
\[\Pi_q:=\{x^2:\> x\in H_q\} = \{ x^*x: x\in H_q\}\]
be the cone of all positive semidefinite
matrices in $H_q$.
Bessel functions  $J_\mu$  on these matrix cones with a parameter $\mu>0$ (and
suppressed parameters $\b F$ and $q$)  were studied from different points of
view by numerous people;
we here only mention  \cite{H}, \cite{FK}, \cite{R}, and \cite{RV} which
are relevant here. As we do not need details, we do not recapitulate the
complicated definition here and refer to these references.
We only mention that for
$q=1,$ and $\b F=\b R$, we have  $\Pi_q=[0,\infty[$,  and the Bessel function $\mathcal J_\mu$ satisfies
\[ J_\mu\bigl(\frac{x^2}{4}\bigr) = j_{\mu-1}(x) \]
where $\, j_\kappa(z) = \,_0F_1(\kappa +1; -z^2/4)\,$
is the usual  modified Bessel function in one variable.

Hypergroups are convolution structures which generalize locally compact
groups insofar as the convolution product of two point measures is
in general not a point measure again, but just a probability measure on the
underlying space. 
More precisely, a hypergroup $(X,*)$ is a locally compact Hausdorff space $X$ together with a convolution $*$
 on the space  $M_b(X)$ of  regular bounded Borel measures on $X$,
such that $(M_b(X),*)$ becomes a Banach algebra, and $*$ is
weakly continuous, probability preserving and  preserves compact
supports of measures.
 Moreover, one requires an identity $e\in X$  
with  $\delta_e*\delta_x=
\delta_x*\delta_e=\delta_x$ for $x\in X$, as well as
  a continuous involution $x\mapsto\bar x$ on $X$ such that for all $x, y\in X$,
$e\in supp(\delta_x*\delta_y)$ is equivalent to $ x=\bar y$,
and
$\delta_{\bar x} * \delta_{\bar y} =(\delta_y*\delta_x)^-$.
Here for $\mu\in M_b(X)$, the measure  $\mu^-$ is given by
$\mu^-(A)=\mu(A^-)$ for Borel sets $A\subset X$.
A hypergroup $(X,*)$ is called commutative if and only if so is the
convolution $*$. Thus for a commutative hypergroup $(X,*)$, 
$M_b(X)$ is a commutative Banach-$*$-algebra with identity $\delta_e$.
  Due to  weak continuity, the convolution of measures on a
  hypergroup is 
uniquely determined by  convolution products of point measures. 

On a commutative hypergroup $(X,*)$ there exists a (up to a multiplicative
factor) unique Haar measure $\omega$, i.e.
 $\omega$ is a positive Radon measure on $X$ satisfying 
\[ \int_X \delta_x*\delta_y(f)d\omega(y)  = \int_X f(y)d\omega(y)  \quad \text{for all } \, x\in X, \, f\in C_c(X).\]
The decisive object for harmonic analysis on a commutative hypergroup is its
dual space
\[ \widehat X :=\{\phi\in C_b(X): \phi\not= 0, \,\phi(\overline x) =
\overline{\phi(x)}, \, 
\delta_x*\delta_y(\phi) = \phi(x)\phi(y) \, \text{ for all }\, x,y\in X\}.\] 
Its elements are  called
 characters. As for LCA groups, the dual of a 
commutative hypergroup is a locally compact Hausdorff space with the topology of locally uniform 
convergence and can be identified with the symmetric spectrum of the convolution algebra $L^1(X,\omega).$ 
For more details on hypergroups we refer to \cite{J} and \cite{BH}.

\medskip
The following theorem contains some of the main results of  \cite{R}.

\begin{theorem}\label{main2}
Let $\mu\in\b R$ with $\mu > \rho-1.$ Then
\parskip=-1pt
\begin{enumerate}\itemsep=-1pt
\item[\rm{(a)}] The assignment
\begin{equation}\label{def-convo}
(\delta_r *_\mu \delta_s)(f) := 
\frac{1}{\kappa_\mu}\int_{D_q} f\bigl(\sqrt{r^2 + s^2 + svr + rv^*\!s}\,\bigr)\,
\Delta(I-vv^*)^{\mu-\rho}\, dv
\end{equation}
for $f\in C(\Pi_q) $ with  $\kappa_\mu$ as in (\ref{kappa}),
  defines a commutative hypergroup structure on $\Pi_q$ with 
neutral element $0\in \Pi_q$ and the identity mapping as involution.
The support of $\delta_r*_\mu\delta_s$ satisfies
\[\text{supp}(\delta_r*_\mu\delta_s) \subseteq  
\{t\in \Pi_q: \|t\|\leq \|r\|+\|s\|\}.\]
\item[\rm{(b)}] A Haar measure of
$(\Pi_q, *_\mu)$ is given by
\[ \omega_\mu(f) = 
\frac{\pi^{q\mu}}{\Gamma_{\Omega_q}(\mu)}
\int_{\Omega_q} f(\sqrt{r}) \Delta(r)^\gamma dr
\quad\quad\text{with}\quad\quad
\gamma = \mu-\frac{d}{2}(q-1)-1.
\]
\item[\rm{(c)}]
The dual space of $\Pi_{q,\mu}$ is given by $\widehat{\Pi_{q,\mu}} = \,\{\phi_s: s\in \Pi_q\}$
with 
$$\phi_s(r):= \mathcal J_\mu(\frac{1}{4}rs^2r)=\phi_r(s).$$
 The hypergroup $\Pi_{q,\mu}$ is self-dual via the homeomorphism $s\mapsto \phi_s$.
Under this identification of $\widehat{\Pi_{q,\mu}}$ with $\Pi_{q,\mu}\,$,
  the Plancherel measure on $\Pi_{q,\mu}\,$ is $(2\pi)^{-2\mu q}\omega_\mu$.
\end{enumerate}
\end{theorem}

The most important informal observation at this point is that the convolution
(\ref{def-convo}) converges for $\mu\to\infty$ to the  semigroup
convolution
$$(\delta_r \bullet \delta_s)(f) :=f(\sqrt{r^2 + s^2}), \quad\quad r,s\in\Pi_q$$
associated with the semigroup operation 
$r\bullet s:= \sqrt{r^2 +  s^2}$ on $\Pi_q$. We next shall make this
convergence more precise, as this is the main ingredient for the proof of
Theorem \ref{CLT4}.

\section{The central limit theorem for $\mu_n>>n$}

We here derive a generalization of Theorem \ref{CLT4}
for general parameters $\mu$. We begin with some unusual notion which is needed
below:

\begin{definition}\label{def-root-lipschitz}
A function $f:\Pi_q\to \b C$ is called root-Lipschitz continuous with constant
$L$, if for all $x,y\in \Pi_q$,
$$| f(\sqrt{x}) -   f(\sqrt{y})|\le L\|x-y\|.$$
\end{definition}

\begin{lemma}\label{lemma1} There is a constant $C=C(q)$ such that for all $r,s\in
  \Pi_q$, all $\mu\ge 2\rho$, and all  root-Lipschitz continuous functions  $f$
  on $\Pi_q$ with constant $L$,
$$|\delta_r *_\mu \delta_s (f) - \delta_r \bullet \delta_s (f)|\le CL\|r\|\cdot \|s\|/\sqrt\mu.$$
\end{lemma}

\begin{proof} We first recapitulate from Lemma 3.1 of  \cite{RV} that for
  $\mu>\rho$ and $v\in \sqrt{\mu}\cdot D_q\subset M_q$,
\begin{equation}\label{exp-inequality}
 0\le e^{-\langle v,v\rangle}-\Delta(I-\frac{1}{\mu}vv^*)^\mu \>\le \>
\frac{1}{\mu} tr((vv^*)^2)\cdot e^{-\langle v,v \rangle}.
\end{equation}
As all norms on $\Pi_q$ are equivalent, we conclude from the first inequality 
 that for $r,s>0$ and  suitable constants $C_i$,
\begin{align}\label{g2}
|\delta_r *_\mu &\delta_s (f) - \delta_r \bullet \delta_s (f)|
\notag\\ &
\le\frac{1}{\kappa_\mu}  \int_{D_q} \biggl| f\bigl(\sqrt{r^2 + s^2 + svr
  + rv^*\!s}\bigr)- f\bigl(\sqrt{r^2 + s^2 }\bigr)\biggr|
\cdot  \Delta(I-vv^*)^{\mu-\rho}\, dv
\notag\\ & 
\le\frac{L}{\kappa_\mu}  \int_{D_q} \|svr +rv^*s\| \cdot
\Delta(I-vv^*)^{\mu-\rho}\, dv
\notag\\ & 
\le\frac{C_1 L\|r\|\cdot \|s\|}{\kappa_\mu}  \int_{D_q} \|v\|\cdot
\Delta(I-vv^*)^{\mu-\rho}\, dv
\notag\\ & = \frac{C_1 L\|r\|\cdot \|s\|}{\kappa_\mu (\mu-\rho)^{(dq^2+1)/2}}
 \int_{ \sqrt{\mu-\rho}\cdot D_q} \|v\|\cdot
\Delta(I-\frac{1}{\mu-\rho}vv^*)^{\mu-\rho}\, dv
\notag\\ & 
\le\frac{C_1 L\|r\|\cdot \|s\|}{\kappa_\mu (\mu-\rho)^{(dq^2+1)/2}}
 \int_{ M_q} \|v\|\cdot e^{-\langle v,v\rangle}\, dv
\notag\\ & 
\le C_2 \frac{L\|r\|\cdot \|s\|}{\kappa_\mu (\mu-\rho)^{(dq^2+1)/2}}
\end{align}
Moreover, the second inequality in (\ref{exp-inequality}) yields that for
sufficiently large $\mu$,
\begin{align}
\kappa_\mu &=  \int_{D_q} \Delta(I-v^*v)^{\mu-\rho} dv
\notag\\
 &=\frac{1}{(\mu-\rho)^{dq^2/2} } \int_{ \sqrt{\mu-\rho}\cdot D_q} 
\Delta(I-\frac{1}{\mu-\rho}vv^*)^{\mu-\rho}\, dv
\notag\\ & 
\ge\frac{1}{(\mu-\rho)^{dq^2/2}}  \int_{ \sqrt{\mu-\rho}\cdot D_q} 
\biggl(1- \frac{ tr((vv^*)^2)}{\mu-\rho} \Biggr)  e^{-\langle v,v\rangle}\, dv
\notag\\ & 
\ge C_3(\mu-\rho)^{-dq^2/2} 
\notag
\end{align}
and thus $\kappa_\mu\ge C_4(\mu-\rho)^{-dq^2/2} $ for all $\mu\ge2\rho$ and 
some constant $C_3$. The lemma is now a consequence of Eq.~(\ref{g2}).
\end{proof}

\begin{lemma}\label{lemma2}
There is a constant $C=C(q)$ such that for all root-Lipschitz continuous functions  $f$
  on $\Pi_q$ with constant $L$, all 
 $\nu_1,\nu_2\in M^1(\Pi_q)$ with 
$m_1(\nu_i)<\infty$, and all  $\mu\ge 2\rho$, 
$$|\nu_1 *_\mu \nu_2(f) - \nu_1 \bullet \nu_2 (f)|\le CL \cdot m_1(\nu_1) \cdot m_1(\nu_2) /\sqrt\mu.$$
\end{lemma}

\begin{proof} By Lemma \ref{lemma1},
\begin{align}
|\nu_1 *_\mu &\nu_2(f) - \nu_1 \bullet \nu_2 (f)| 
\notag\\ & 
\le \int_{\Pi_q}\int_{\Pi_q} |\delta_r *_\mu \delta_s (f) - \delta_r \bullet
\delta_s (f)|\> d\nu_1(r)\> d\nu_2(s)
\notag\\ & 
\le \frac{CL}{\sqrt\mu} \int_{\Pi_q}\int_{\Pi_q} \|r\|\cdot \|s\| \> d\nu_1(r)\>
d\nu_2(s)
\quad=\quad CL  \frac{m_1(\nu_1)  m_1(\nu_2)}{\sqrt\mu}.
\notag\end{align}
\end{proof}

\begin{lemma}\label{lemma3}
There is a constant $C=C(q)$ such that for all root-Lipschitz continuous functions  $f$
  on $\Pi_q$ with constant $L$, all 
 $\nu_1,\nu_2, \nu_3\in M^1(\Pi_q)$ with 
$m_1(\nu_i)<\infty$, and all  $\mu\ge 2\rho$, 
$$|(\nu_1 *_\mu \nu_2)\bullet\nu_3 (f) - (\nu_1 \bullet \nu_2)\bullet\nu_3
(f)|\le
 CL \cdot m_1(\nu_1) \cdot m_1(\nu_2) /\sqrt\mu.$$
\end{lemma}

\begin{proof}
For $y\in\Pi_q$, consider the function $f_y(x):=f(x\bullet y)$ on $\Pi_q$. By
the definition of $\bullet$ and our definition of root-Lipschitz continuity, these $f_y$
are also  root-Lipschitz continuous with the same constant $L$. Therefore, by
Lemma  \ref{lemma2},
\begin{align}
|(\nu_1 *_\mu& \nu_2)\bullet\nu_3 (f) - (\nu_1 \bullet \nu_2)\bullet\nu_3 (f)|
\notag\\ & 
\le\int_{\Pi_q}\biggl| \int_{\Pi_q} f_y(x) \> d(\nu_1 *_\mu \nu_2)(x) - \int_{\Pi_q} f_y(x) \> d(\nu_1 \bullet \nu_2)(x)
\biggr| \> d\nu_3(y)
\notag\\ & 
\le CL \cdot m_1(\nu_1) \cdot m_1(\nu_2) /\sqrt\mu.
\notag\end{align}
\end{proof}

For $\nu\in M^1(\Pi_q)$ and $n\in\b N$, we denote the $n$-fold convolution
powers
of $\nu$ w.r.t.~the convolutions $*_\mu$ and $\bullet$ by 
$\nu^{(n,*_\mu)}$ and $\nu^{(n,\bullet)}$ respectively.

\begin{lemma}\label{lemma4}
For all  $\nu \in M^1(\Pi_q)$ with
$m_2(\nu)<\infty$, and all $n\in\b N$, $\mu>\rho-1$,
$$m_1(\nu^{(n,*_\mu)})\le \sqrt{n\cdot m_2(\nu)}.$$
\end{lemma}

\begin{proof}
By the definition of the convolution $*_\mu$, the function $m_2$ satisfies
$$\delta_r*_\mu\delta_s(m_2)= m_2(r)
+m_2(s)+\frac{1}{\kappa_\mu}\int_{D_q}tr(rvs+sv^*r)\cdot 
\Delta(I-vv^*)^{\mu-\rho}\, dv$$
for $r,s\in\Pi_q$ where the symmetry of the integrand and the substitution
$v\mapsto -v$ immediately yield that the integral is equal to 0.
Therefore,
$$\delta_r*_\mu\delta_s(m_2)= m_2(r)+m_2(s).$$
Integration yields that for all $\nu_1,\nu_2 \in M^1(\Pi_q)$ with
$m_2(\nu_i)<\infty$, 
$$m_2(\nu_1*_\mu\nu_2)=m_2(\nu_1)+m_2(\nu_2).$$
In particular we obtain by induction that for $\nu\in M^1(\Pi_q)$ with
 $m_2(\nu)<\infty$ and $n\in\b N$, 
$$m_2(\nu^{(n,*_\mu)})=n\cdot m_2(\nu).$$
Therefore, by the Cauchy-Schwarz inequality,
$$m_1(\nu^{(n,*_\mu)})=\int_{\Pi_q}\|x\|\>  d\nu^{(n,*_\mu)}(x)\le \biggl(\int_{\Pi_q}\|x\|^2\>  d\nu^{(n,*_\mu)}(x)\biggr)^{1/2}
=\sqrt{n\cdot m_2(\nu)}.$$
\end{proof}

\begin{remark} The preceding lemma has the following weaker variant under a
  weaker moment condition: For all  $\nu_1,\nu_2 \in M^1(\Pi_q)$ and all $\mu>\rho-1$,
$$m_1(\nu_1 *_\mu \nu_2)\le m_1(\nu_1)+m_1(\nu_2).$$

For the proof, observe that the convolution $*_\mu$ has the property that for $r,s\in\Pi_q$,
$$supp(\delta_r*_\mu\delta_s)\subset \{t\in\Pi_q:\quad \|t\|\le \|r\| +\|s\|
\};$$
see for instance Theorem 3.10 of \cite{R}. Therefore,
\begin{align}
m_1(\nu_1 *_\mu \nu_2) &= 
\int_{\Pi_q}\int_{\Pi_q}\biggl( \int_{\Pi_q} \|r\|  \> d(\delta_x *_\mu
\delta_y)(r)\biggr) \>  d\nu_1(x) \> d\nu_2(y)
\notag\\ & 
\le \int_{\Pi_q}\int_{\Pi_q}(\|x\| +\|y\|) \> d\nu_1(x) \> d\nu_2(y)
\notag\\ & 
= m_1(\nu_1)+m_1(\nu_2)
\notag\end{align}
as claimed.
\end{remark}

\begin{proposition}\label{proposition1}
Let $ \nu\in M^1(\Pi_q)$ with $m_2(\nu)<\infty$. Then
there is a constant $C=C(q,\nu)$ such that for all root-Lipschitz continuous functions  $f$
  on $\Pi_q$ with constant $L$, and all $n\ge2$, $\mu\ge 2\rho$, 
$$\bigl| \nu^{(n,*_\mu)}(f) - \nu^{(n,\bullet)}(f)\bigr| \le CL  \frac{n^{3/2}}{\sqrt\mu}.$$
\end{proposition}

\begin{proof}
We first observe that
\begin{align}
&\bigl| \nu^{(n,*_\mu)}(f) - \nu^{(n,\bullet)}(f)\bigr| 
\notag\\ & 
\le    
\bigl| \nu^{(n,*_\mu)}(f) - \nu^{(n-1,*_\mu)}\bullet\nu(f)\bigr| +
\bigl|(\nu^{(n-2,*_\mu)}*\nu)\bullet\nu(f) - \nu^{(n-2,*_\mu)}\bullet\nu\bullet\nu(f)\bigr|
\notag\\ & \quad +\ldots+ 
\bigl|(\nu*\nu)\bullet \nu^{(n-2,\bullet)}(f) - \nu^{(n,\bullet)}(f)\bigr|.
\notag\end{align}
Moreover, by Lemmas \ref{lemma3} and \ref{lemma4}, we have for $k=2,\ldots,n$
that
\begin{align}
\bigl|\nu^{(k,*_\mu)}\bullet \nu^{(n-k,\bullet)}(f)&- \nu^{(k-1,*_\mu)}\bullet\nu\bullet \nu^{(n-k,\bullet)}(f)\bigr|
\notag\\ & 
\le     m_1(\nu^{(k-1,*_\mu)}) \cdot m_1(\nu)\cdot \frac{C_1 L}{\sqrt\mu}
\notag\\ & 
\le  
C_1L \frac{\sqrt{k-1}}{\sqrt\mu}m_1(\nu)\cdot \sqrt{m_2(\nu)}
\notag\end{align}
with a suitable constant $C_1>0$. Combining this with the preceding inequality
and $\sum_{k=1}^n \sqrt k= O(n^{3/2})$, the proposition follows. 
\end{proof}

We now fix a probability measure $\nu\in M^1(\Pi_q)$ and 
  consider for  $\mu>\rho$,
the associated random walk $(S_n^{\mu})_{n\in\b N}$ on $\Pi_q$ with law $\nu$
according to Definition \ref{random-walk}.

\begin{theorem}\label{theorem-mu-ge-n}
Let $\nu\in M^1(\Pi_q)$ with a finite fourth moment $\int_{\Pi_q} \|x\|^4 \>
d\nu(x)<\infty$.
 Assume that $n^2/\mu_n\to 0$ for $n\to\infty$. Then
$$\frac{(S_n^{\mu_n})^2 -n\sigma^2}{\sqrt n}$$
tends in distribution for $n\to\infty$ to the normal distribution 
$N(0,\Sigma^2)$ on the vector space $H_q$ of Hermitian $q\times q$ matrices, where
$\Sigma^2$ is the covariance matrix belonging to the image measure $Q(\nu)\in
M^1(\Pi_q)\subset M^1(H_q) $ under the square mapping $Q(x):=x^2$ on $\Pi_p$.
\end{theorem}

\begin{proof}
Let $f\in C_0(H_q)$ be a Lipschitz-continuous function in the usual sense on
$H_q$
with the Lipschitz constant $L$. Then, for $n\in\b N$, the functions 
$$f_n(x):= f\bigl(\frac{x^2-n\sigma^2}{\sqrt n}\bigr)$$
on $\Pi_q$ are root-Lipschitz with constants $n^{-1/2}L$. Therefore, by
Proposition \ref{proposition1} and the assumptions of the theorem,
\begin{equation}\label{g3}
\bigl| \int_{\Pi_q} f_n \> d\nu^{(n,*_{\mu_n})} - \int_{\Pi_q} f_n \>
d\nu^{(n,\bullet)}\bigr|
  \le CL \frac{n}{\sqrt\mu_n} \to 0
\end{equation}
for $n\to\infty$ with
\begin{equation}\label{g4}
\int_{\Pi_q} f_n \> d\nu^{(n,*_{\mu_n})} =\int_{\Pi_q} f_n \> dP_{S_n^{\mu_n}}=
\int_{H_q} f dP_{((S_n^{\mu_n})^2 -n\sigma^2)/\sqrt n},
\end{equation}
where $P_X$ denotes the law of a random variable $X$.
Moreover, as the square mapping $Q(x):=x^2$ is a isomorphism from the semigroup
$(\Pi_q,\bullet)$ onto the semigroup $(\Pi_q,+)$, and denoting the classical
convolution of measures on $\Pi_q\subset H_q$ associated with the operation $+$
by $*$, we see that
\begin{align}\label{g5}
\int_{\Pi_q} f_n \>
d\nu^{(n,\bullet)} &= \int_{\Pi_q} f\biggl(\frac{x^2-n\sigma^2}{\sqrt n}\biggr)
\> d\nu^{(n,\bullet)}(x)
\notag\\ & =\int_{\Pi_q} f\biggl(\frac{x-n\sigma^2}{\sqrt n}\biggr)
\> dQ(\nu^{(n,\bullet)})(x)
\notag\\ &=\int_{\Pi_q} f\biggl(\frac{x-n\sigma^2}{\sqrt n}\biggr)
\> d(Q(\nu)^{(n,*)})(x)
\end{align}
where the latter tends by the classical central limit theorem on the euclidean
space $H_q$ to $\int_{H_q} f\> dN(0,\Sigma^2)$. Taking (\ref{g3}) and
(\ref{g4}) into account, we obtain that
\begin{equation}\label{g6}
\int_{H_q} f dP_{((S_n^{\mu_n})^2 -n\sigma^2)/\sqrt n} \to \int_{H_q}  f\> dN(0,\Sigma^2)
\end{equation}
for $n\to\infty$. As the space of Lipschitz continuous functions in $C_0(H_q)$ is
$\|.\|_\infty$-dense in $C_0(H_q)$, a simple $\epsilon$-argument together with
the triangle inequality ensures that Eq.~(\ref{g6}) holds for all $f\in
C_0(H_q)$. This completes the proof.
\end{proof}

We briefly consider the  case of point measures $\nu=\delta_x$.
In this case we have  $\Sigma^2=0$ and we do not need to apply the classical CLT
above. In particular, the first part of the preceding proof
leads to the following weak law:

\begin{corollary}
Let $(a_n)_{n\ge1}\subset ]0,\infty[$ an increasing sequence with 
$a_n= o\bigl(\frac{n^{3/2}}{\mu^{1/2}}\bigr)$ for $n\to\infty$. 
If $\nu=\delta_x$  for some $x\in \Pi_q$,  then 
$$\frac{(S_n^{\mu_n})^2 -nx^2}{a_n}\to 0 \quad\quad\text{in probability}.$$ 
\end{corollary}

In particular, for $a_n:=n$ and $n/\mu_n\to 0$, we obtain $S_n^{\mu_n}/\sqrt n\to
x$. We note that this particular result was derived under much weaker conditions on the $\mu_n$ in
\cite{RV} by different methods.

We finally note that a similar CLT is derived in \cite{V3}
 for radial random walks on the hyperbolic spaces when time and dimension tend to infinity
 similar as above.

\end{document}